\documentclass[12pt]{article}
\usepackage{amssymb,latexsym, amsmath, enumerate}
\usepackage{color}
\usepackage{amsthm}
\usepackage{graphicx}
\usepackage{subfigure}

\newtheorem{theorem}{Theorem}
\newtheorem{example}{Example}
\newtheorem{definition}{Definition}
\newtheorem{lemma}{Lemma}

\newtheorem{corollary}{Corollary}
\newtheorem{proposition}{Proposition}
\newtheorem{notation}{Notation}

\title{Besicovitch-Morse type  covering lemmas\\ in metric spaces}
\author{Tong Zhang}
\date{}

\begin{document}

\renewcommand{\theequation}{\thesection.\arabic{equation}}
\setcounter{equation}{0} \maketitle

\vspace{-1.0cm}

\bigskip

\!\!\!\!\!\!\!\!{\bf Abstract.} The aims of this article is to generalize some useful Besicovitch-Morse type covering lemmas in complete Riemannian manifolds and try to find more spaces such that the so-called BCP and WBCP are equivalent while these two properties are weaker and still useful. We also get interest in the best constants of Besicovitch-type covering properties in Euclidean spaces and sorted out the best results of related problems before giving a new proof of Besicovitch covering theorem in the one-dimensional case.

\medskip

\!\!\!\!\!\!\!\!{\bf Key words:} covering lemmas, BCP, WBCP, Riemannian manifold, Metric spaces

\medskip

\renewcommand{\theequation}{\thesection.\arabic{equation}}

\section{Introduction}

\ \ \ \ Covering theorems are fundamental tools of measure theory and classical analysis among the past several decades. One of the most important covering theorems is the Besicovitch covering theorem which is first established for the plane by A.S.Besicovitch[4][5] and extended by A.P.Morse[22] to more general sets in finite dimensional normed vector spaces. H.Federer[13] has shown the validity of it in some special metric spaces. The Besicovitch covering theorem has some weaker versions which are more convenient to check. Here we consider the so-called Besicovitch covering property, BCP for short, and weak Besicovitch covering property, WBCP for short, while they are named and investigated by S.Rigot[8][9][10][14][24] in Lie groups. The validity of BCP in a metric space can deduce the Lebesgue Differentiation Theorem and the validity of WBCP can deduce the uniform weak type (1,1) of the centered maximal operator. In[1], J.M.Aldaz investigates more Besicovitch type properties in some special metric spaces including doubling spaces and ultrametric spaces. J.Jost, V.L.Hong. and T.T.Tran[16] give a proof of a Besicovitch-type covering theorem in complete Riemannian manifolds.\\

In Section 2, we will give some standard notions and definitions used in the next proofs.
In Section 3, we will use a more essential way to prove the Besicovich covering theorem in complete Riemannian manifolds. At first, we show that the validity of  BCP implies the validity of Besicovitch covering theorem. Then we also show that the validity of  WBCP implies the validity of BCP. Hence we just need to check the validity of WBCP in complete Riemannian manifolds which is much easier and then we also deduce the differentiation theorem for any Radon measure in complete Riemannian manifolds. And then we generalize the Besicovitch-Morse type covering theorems to complete Riemannian manifolds. In section 4, we introduce a new property of metric spaces and then prove that BCP and WBCP are equivalent in any metric spaces satisfying this property. This property is not difficult to satisfy and we will give a proof of the validity of this property in finite dimensional Euclid spaces. In section 5, we sorted out the best results of the best constants of the Besicovitch-type properties in Euclidean spaces which are related to some interesting problems in geometry. Here we give a new proof of the best constant of Besicovitch covering theorem in the one-dimensional case.

\section{Some standard notions and  definitions}

\ \ \ \ Throughout the paper $(X,d)$ is a metric space. $B({x},{r})$ denotes a closed metric ball with center $x\in X$ and radius $r\in \mathbb{R}$ and $\partial B({x},{r})$ denotes the surface of $B({x},{r})$. \\

$(M,\rho)$ is a Riemannian manifold. The notation $in{j_x}M$ denotes the injectivity radius of $x\in M$ and $exp_a$ denotes the exponential map with respect to $x$.\\ 

${{\chi _A}}$ denotes the characteristic function of the set $A$.\\

$diamA$ denotes the diameter of the set $A$.\\

$CardA$ denotes the cardinality of the set $A$.\\

$min\{A,B\}$ denotes the minimum of the elements of the set $\{A,B\}$\\

${\mathbb{R}^n}$ denotes $n$-dimensional real Euclidean space.

\begin{definition}
	\rm A collection of balls ${\cal B} = \{ B({x_i},{r_i})\} _{i = 1}^n$ in a metric space $(X,d)$ is said to be  a \it Besicovitch family \rm if \[\bigcap\limits_{i = 1}^n {B({x_i}} ,{r_i}) \ne \emptyset \] and for every pair of distinct balls $B({x_i},{r_i})$, $B({x_j},{r_j}) \in \cal B $, we have ${x_j}\notin B({x_i},{r_i})$ and ${x_i}\notin B({x_j},{r_j})$.\\
\end{definition}

\begin{definition}
	\rm (BCP) A metric space $(X,d)$ has \it the Besicovitch covering property \rm if there exists a constant C such that for every $R>0$, every set (or bounded set for some versions) $A\subset X$ and every collection of balls \[{\cal F}=\{B(x,r): x\in A, r \leqslant R\},\] such that each point of $A$ is the center of some ball of $\mathcal{F}$,
	we can find a subcollection $\mathcal{F}'\subset {\cal F} $ such that \[{\chi _A} \leqslant \sum\limits_{B \in \mathcal{F}'} {{\chi _B}}  \leqslant C.\]
\end{definition}

\begin{definition}
	\rm (WBCP) A metric space $(X,d)$ has \it the weak Besicovitch covering property \rm if there exists a constant $L$ such that $card\ \mathcal{B}\leqslant L$ for every Besicovitch family $\mathcal{B}$.
\end{definition}

\begin{definition}\rm ($\epsilon$-net)
	$(X,d)$ is a metric space. A \it$\epsilon$-net(resp. strict $\epsilon$-net) \rm is a subset $S\subset X$ such that for any pair of distinct points $x,y\in S$, we have $d(x,y)\geqslant r$ (resp. $d(x,y)>r$).
\end{definition}

\section{Covering properties in complete Riemannian manifolds}

\subsection{A new proof of Besicovitch-type theorem in complete Riemannian manifolds}

\begin{theorem}
\rm $(M,\rho) $ is a complete Riemannian manifold and $A$ is a bounded subset of $M$. There exists ${C_3}(A)\in \mathbb{N}$ such that for arbitrary collection of geodesic balls \[{\cal F}=\{ B({x},{r}):{x} \in A, {r} \leqslant \frac{1}{4}\mathop {\inf }\limits_{x \in A} in{j_x}M\}\] and $S$ denoting the centers of the balls in $\cal F$, we can find a $N \in \mathbb{N}$ not exceeding  ${C_3}(A)$ and subcollections ${\mathcal{F}_1},{\mathcal{F}_2}, \cdots ,{\mathcal{F}_N}$ such that $S$ is covered by the balls of these subcollections and each subcollection consists of at most countably many disjoint balls.                
\end{theorem}

Theorem 1 is obvious after we attain Proposition 1 , Proposition 2 and Theorem 2.\\ 

\begin{proposition}
	\rm	(BCP $\Rightarrow$ BC Thm)$  \ (M,\rho)$ is a complete Riemannian manifold and $A$ is a bounded subset of $M$. If there exists a constant ${C_2}(A)$ such that for arbitrary collection of balls \[{\cal F}=\{B(x,r): x\in A, r \leqslant \dfrac{1}{4}\mathop {\inf }\limits_{x \in A} in{j_x}M\}\] and $S$ denoting the centers of the balls in $\cal F$,
	we can find a subcollection $\mathcal{F}'\subset {\cal F} $ such that \[{\chi _S} \leqslant \sum\limits_{B \in \mathcal{F}'} {{\chi _B}}  \leqslant {C_2}(A).\]
	
	Then there will exists a constant ${C_3}(A)$ such that we can find a $m \in \mathbb{N}$ not exceeding ${C_3}(A)$ and $m$ subcollections  ${\mathcal{F}_1},{\mathcal{F}_2}, \cdots ,{\mathcal{F}_m}$ of disjoint balls such that\[S \subset \bigcup\limits_{i = 1}^m {\bigcup\limits_{B \in {\mathcal{F}_i}} B }.\]
\end{proposition}

\begin{lemma}
	\rm $(M,\rho)$ is a complete Riemannian manifold. For any geodesic ball $B(x,r)$ with $r<{{in{j_x}M}}$, $y\in B(x,r)$	and $0<s<r$, we can find a $z\in B(x,r)$ such that \[y \in B(z,s) \subset B(x,r).\]
\end{lemma}

\begin{proof}[\textbf{\textit{Proof of lemma 1}}]
Here $\gamma$ denotes the minimal geodesic connected by $x$ and $y$ with $\gamma (0) = y$ and $\gamma (1) = x$. Thus $\rho(x,y)=Length\ \gamma$ and there exists ${t_0}$ and $z\in \gamma$ such that 

\[\gamma ({t_0}) = z,\ \rho(y,z)=s,\  \rho(x,z)=\rho(x,y)-s.\]

For every $w\in B(z,s)$,
\[\rho (w,x) \leqslant \rho (w,z) + \rho (z,x) \leqslant s + \rho (x,y) - s = \rho (x,y) \leqslant r.\]
\end{proof}

\begin{definition} 
	\rm For an arbitrary ball $B(x,r)$, a collection of balls $\{ B({x_i},{r_i})\} _{i = 1}^n$ is said to be an $\alpha$-\textit{configuration} with respect to $B(x,r)$ if\[B({x_i},{r_i}) \cap B(x,r) \ne \emptyset \quad and   \quad r < \alpha {r_i}\] for each $i \in\{ 1,2, \cdots ,n\}$.\\	
\end{definition}

\begin{proof}[\textbf{\textit{Proof of Proposition 1}}]
    Fix a collection of balls\[{\cal F}=\{B(x,r): x\in A, r \leqslant \dfrac{1}{4}\mathop {\inf }\limits_{x \in A} in{j_x}M\}.\]$S$ denotes the centers of the balls in $\cal F$. Then we can find a subcollection $\mathcal{F}'\subset {\cal F} $ such that \[{\chi _S} \leqslant \sum\limits_{B \in \mathcal{F}'} {{\chi _B}}  \leqslant {C_2}(A).\]
    For $\forall\  \alpha  \in (\dfrac{1}{2},1)$, set \[{\mathcal{D}^{{n}}}{{ = \{ B(x,r)}} \in \mathcal{F}', r \in ({\alpha ^{n}}R,{\alpha ^{n-1}}R]{\text{\} }},\] while $R$ denotes $\dfrac{1}{4}\mathop {\inf }\limits_{x \in A} in{j_x}M.$  \\ 
    
    Choose a largest family of disjoint balls in ${\mathcal{D}^{{1}}}$ and then  inductively choose a largest family of disjoint balls in ${\mathcal{D}^{{n}}}$  such that each ball of this family does not intersect any chosen ball in $\{ {\mathcal{D}^i}\} _{i = 1}^{n - 1}$ for $n>2$.
    ${\mathcal{F}^{{1}}}$ denotes all the chosen balls in $\{ {\mathcal{D}^i}\} _{i = 1}^{\infty}$.\\
    
    After we get a collection of disjoint balls ${\mathcal{F}^{{1}}}$, repeat the above procedure for the left balls in $\mathcal{F}'$ to get  $\{\mathcal{F}^i\} _{i = 2}^{\infty}$ inductively. Suppose that $\{\mathcal{F}^i\} _{i = 1}^{k}$ have been defined and\[\mathcal{F}'\backslash \bigcup\limits_{i = 1}^k {{F_i}}  \ne \emptyset .\]  
    
    Then there must be a ball $B(x,r)\in F\backslash \bigcup\limits_{i = 1}^k {{F_i}}$ and thus we can find $B(x_i,r_i)\in \mathcal{F}^i$ such that \[B(x_i,r_i)\cap B(x,r)\ne \emptyset \ \ and\ \  r_i>\dfrac{1}{\alpha}r\] for each $i$. This is because there must be some $n$ such that $s\in({\alpha ^{n}}R,{\alpha ^{n-1}}R]$ and then we can find a $B(x_i,r_i)\in \mathcal{D}^j$ for some $j<n+1$ such that $B(x_i,r_i)\cap B(x,r)\ne \emptyset$. \\
    
    Thus $\{ B({x_i},{r_i})\} _{i = 1}^k$ is an $\alpha$-configuration with respect to $B(x,r)$.\\
    
    \!\!\!\!\!\!\!\!\!\!\it \textbf{Claim 1:} There exists a $C_{\alpha}(A)$ such that the cardinality of a strict $\dfrac{r}{\alpha}$-net in $B(x,r+\dfrac{r}{\alpha})$ can not exceed $C_{\alpha}(A)$ for any ball $B(x,r)\in \mathcal {F}$.\\

    \!\!\!\!\!\!\!\!proof of claim 1 : \rm For an strict $\dfrac{r}{\alpha}$-net $\{ {x_1},{x_2}, \cdots ,{x_k}\} $ in $B(x,r+\dfrac{r}{\alpha})$, we consider a collection of disjoint balls $\{ {B(x_1,\dfrac{r}{2\alpha})},{B(x_2,\dfrac{r}{2\alpha})}, \cdots ,{B(x_k,\dfrac{r}{2\alpha})}\} $ contained in $B(x,r+\dfrac{3r}{2\alpha})$.
    
    Since $\alpha  \in (\dfrac{1}{2},1)$ and $r \leqslant \dfrac{1}{4}\mathop {\inf }\limits_{x \in A} in{j_x}M$, we have \[r+\dfrac{3r}{2\alpha}\leqslant \mathop {\inf }\limits_{x \in A} in{j_x}M.\] Thus there exists two constants $M(A)$ and $N(A)$ such that  \[vol B(x,r+\dfrac{3r}{2\alpha})\leqslant M(A)\cdot {{(r+\dfrac{3r}{2\alpha}})}^{dim(M,\rho)}\] 
    and \[vol B(x_i,\dfrac{r}{2\alpha}) \geqslant  N(A)\cdot {{(\dfrac{r}{2\alpha}})}^{dim(M,\rho)} \ for \ \ \forall \ i. \]  by the Bishop Comparison theorem. Here
    $vol B$ denotes the Riemannian volume of the set $B$ and
    $dim(M,\rho)$ denotes the dimension of the Riemannian manifold $(M,\rho)$ \\
    
    Thus \[k \leqslant {(2\alpha  + 3)^{dim(M,\rho)}}\]
    This completes the proof of Claim 1. \\
    
    \!\!\!\!\!\!\!\!\!\!\it \textbf{Claim 2:} There exists a $C(A)$ such that the cardinality of any $\alpha$-\textit{configuration} in $\mathcal{F}'$ can not exceed $C(A)$ . \\
    
    \!\!\!\!\!\!\!\!proof of claim 2 : \rm For an arbitrary $\alpha$-\textit{configuration} ${{\{ B(}}{{{x}}_i}{{,}}r_i{{)\} }}_{i = 1}^k$ with respect to $B(x,r)$ in $\mathcal{F}'$, we can find a ball $B(z_i,\dfrac{r}{\alpha})\subset B(x_i,r_i)$ such that $B(z_i,\dfrac{r}{\alpha})\cap B(x,r)\ne \emptyset$ for each $i$ because of lemma 1 and then \[\rho ({z_i},x) < (1 + \frac{1}{\alpha })\ r\ \ for \ \ \forall\ i .\]
    
    For each $i$, we make every ball in ${{\{ B(}}{{{z}}_i}{{,}}\frac{r}{\alpha }{{)\} }}_{i = 1}^k$ containing $z_i$ absorbed into $B(z_i,\dfrac{r}{\alpha})$ and the number of the absorbed balls  is at most $C_2(A)$. At last, there will be a family of balls whose centers consist of a strict $\dfrac{r}{\alpha}$-net in $B(x,r + \frac{1}{\alpha }r)$ left. And the cardinality of it can not exceed $C_\alpha(A).$
    Thus the original $\alpha$-\textit{configuration} with respect to $B(x,r)$ consists of at most  $C_2(A) \cdot C_\alpha(A)$ balls.\\
    
    This completes the proof of Claim 2. \\
    
    Then there exists a $m \in \mathbb{N}$ not exceeding $C_2(A) \cdot C_{\alpha}(A)+1$ and $m$ subcollections  ${\mathcal{F}_1},{\mathcal{F}_2}, \cdots ,{\mathcal{F}_m}$ of disjoint balls such that\[A \subset \bigcup\limits_{i = 1}^m {\bigcup\limits_{B \in {\mathcal{F}_i}} B }.\]
\end{proof}

\begin{proposition}
    \rm	(WBCP $\Rightarrow$ BCP)\ $(M,\rho) $ is a complete Riemannian manifold and $A$ is a bounded subset of $M$. Suppose that there exists ${C_1}(A)\in [1, \infty)$ such that $Card\;\cal B \leqslant$$ {C_1}(A)$ 
    \\\\for every  Besicovich family $\cal B$ $=\{ B({x_i},{r_i}):{x_i} \in A, {r_i} \leqslant \dfrac{1}{4}\mathop {\inf }\limits_{x \in A} in{j_x}M\}$ in $(M,\rho)$.\\\
    
    Then for any collection of geodesic balls \[{\cal F}=\{ B({x},{r}):{x} \in A, {r} \leqslant \dfrac{1}{4}\mathop {\inf }\limits_{x \in A} in{j_x}M\}\]and $S$ denoting the centers of the balls in $\cal F$,
    we can find a subcollection $\mathcal{F}'\subset\mathcal{F}$ such that \[{\chi _S} \leqslant \sum\limits_{B \in \mathcal{F}'} {{\chi _B}}  \leqslant {C_2}(A).\]  Besides, ${C_2}(A)={C_1}(A)\cdot {C_0}(A)+1$ and ${C_0}(A)$ denotes the largest cardinality of strict $r$-net in $B(x,4r)$ for $x\in A$ and $r<\dfrac{1}{4}\mathop {\inf }\limits_{x \in A} in{j_x}M$.
\end{proposition}

\begin{proof}[\textbf{\textit{Proof of Proposition 2}}]
	
	\!\!\!\!\!\!\!\!\!\!\it \textbf{\ \ \ \ \ Claim:} For any subcollection\[{\cal F}^i=\{ B(x,r)\in{\cal F} :x \in S,\ ({\dfrac{1}{2}})^{i}R\leqslant r \leqslant({\dfrac{1}{2}})^{i-1}R\}\]
	while $R$ denotes $\dfrac{1}{4}\mathop {\inf }\limits_{x \in A} in{j_x}M$, we can find a $t_i\leqslant C_0(A)+1$ and $t_i$ subcollections $\{{{\cal F}_j}^i\} _{j = 1}^{t_i}$ of disjoint balls such that \[S_i \subset \bigcup\limits_{j = 1}^{t_i} {\bigcup\limits_{B \in {{\cal F}_j}^i} B }\] where $S_i$ denotes the centers of the balls in ${\cal F}^i$.
	\\ 
	\!\!\!\!\!\!\!\!proof of claim: \rm Choose a largest family of disjoint balls ${{\cal F}_1}^i=\{ B(x_m^1,r_m^1)\} _{m = 1}^{{l_1}}$ in $\mathcal{F}^i$ $(l_1\leqslant\infty)$.
	Inductively, \[\mathcal{T}^k=\{ B(x,r)\in\mathcal{F}^i :x\in S_i\backslash\bigcup\limits_{j=1}^{k-1}\bigcup\limits_{B \in {\mathcal{F}_j}^i}^{} B\}\]
	and choose a largest family of disjoint balls ${\mathcal{F}_k}^i=\{ B(x_m^k,r_m^k)\} _{m = 1}^{{l_k}}$ in $\mathcal{T}^k$ $(l_k\leqslant\infty)$.
	
	If $S_i\backslash\bigcup\limits_{t=1}^{k-1}\bigcup\limits_{B \in {\mathcal{T}^t}}^{} B \ne \emptyset$, there must be a ball $B(x,r)\in \mathcal{F}^i$ such that we can find a $B(x_t,r_t)\in {\mathcal{T}^t}$ for each $t\leqslant k-1$ and $B(x_t,r_t)\cap B(x,r) \ne \emptyset$. 
	
	Thus we have $\rho(x_t,x)\leqslant\dfrac{R}{{2^{i - 2}}}$ and $\rho(x_t,x_j) \geqslant \dfrac{R}{{2^{n }}}$ for each $t,j\leqslant k-1$ and then $\{ {x_1},{x_2}, \cdots ,{x_{k - 1}}\}$ is a $\dfrac{R}{{2^{i }}}$-net in $B(x,\dfrac{R}{{2^{i-2}}})$ which indicates that $k-1\leqslant C_0$.\\
	
	This completes the proof of Claim. \\
	
	Set $\mathcal{D}^1=\mathcal{F}^1$, $E_1$ denotes the centers of balls in $\mathcal{D}^1$ and $S_i$ denotes the centers of balls in $\mathcal{F}^i$. Then we can find subcollections $\{{{\cal F}_j}^1\} _{j = 1}^{t_1}$ of disjoint balls such that \[S_1 \subset \bigcup\limits_{j = 1}^{t_1} {\bigcup\limits_{B \in {{\cal F}_j}^1} B }.\]\\
	
	Inductively, \[\mathcal{D}^k=\{B(x,r)\in\mathcal{F}^k: x\in{S_k}\backslash \bigcup\limits_{i = 1}^{k - 1} {\bigcup\limits_{j = 1}^{{t_i}} {\bigcup\limits_{B \in F_j^i}^{} B } } \}\]
	and $E_k$ denotes the centers of the balls in $\mathcal{D}^k$. Then we can find subcollections $\{{{\cal F}_j}^k\} _{j = 1}^{t_k}$ of $\mathcal{D}^k$ as in the proof of claim such that \[E_k \subset \bigcup\limits_{j = 1}^{t_k} {\bigcup\limits_{B \in {{\cal F}_j}^k} B}\]
	while each subcollection consists of disjoint balls.\\
	
	Set ${\cal F}'=\bigcup\limits_{i = 1}^\infty  {\bigcup\limits_{j = 1}^{{m_i}} {\mathcal{F}_j^i} }$, then $S$ is covered by the balls in ${\cal F}'$ and we just need to prove that \[\sum\limits_{B \in \mathcal{F}'} {{\chi _B}}  \leqslant {C_2}(A).\]
	
	From the way of choosing $\bigcup\limits_{j = 1}^{{m_i}} {\mathcal{F}_j^i}$ and $\bigcup\limits_{j = 1}^{{m_k}} {\mathcal{F}_j^k}$ for $i<k$, we have \[\rho ({x_i},{x_k}) > {r_i} > {r_k}\ \ for\ \forall \ \ B(x_i,r_i)\in \bigcup\limits_{j = 1}^{{m_i}} {\mathcal{F}_j^i} , B(x_k,r_k)\in \bigcup\limits_{j = 1}^{{m_k}} {\mathcal{F}_j^k}\ .\]
	
	Thus for any $y\in X$, $\mathcal{F}_{y}$ denots all balls from $\mathcal{F}$ which contains $y$ and there are at most $C_{1}$ indexes  $\{ {y_1},{y_{2,}} \cdots ,{y_n}\}$  such that \[{{\cal F}^{{y_i}}} \cap {{\cal F}_y} \ne \emptyset .\]
	And for each $i$, $y$ belongs to at most $C_0(A)+1$ balls in ${{\cal F}^{{y_i}}}$. 
	
\end{proof}

\begin{theorem}
    \rm	$(M,\rho) $ is a complete Riemannian manifold and $A$ is a bounded subset of $M$. There exist ${C_1}(A)\in [1, \infty)$ and $R(A)\in (0,+\infty ]$ such that $Card\;\cal B \leqslant $  ${C_1}(A)$ for every \\\\
    \!\!\!\!\!\!\!\!Besicovitch family $\cal B$ $=\{B({x_i},{r_i}):{x_i} \in A, {r_i} \leqslant \dfrac{1}{4}\mathop {\inf }\limits_{x \in A} in{j_x}M\}$ in $(M,\rho)$.
\end{theorem}

\begin{proof}[\textbf{\textit{Proof}}]
	Suppose $y\in\bigcap\limits_{i = 1}^k {B({x_i}} ,{r_i})$ and ${x_i} \notin B({x_j} ,{r_j}) $ for any distinct $i,j\in \{1,2,\cdots,k\}$.
	
	By the Bishop comparison theorem, there exists a $M(A)$ such that\[{\rho ^2}({x_i},{x_j}) \leqslant {{M}^2(A)} \cdot [r_i^2 + r_j^2 - 2{r_i}{r_j}cos{\theta _y}({x_i},{x_j})].\]
	If 
	${\rho}({x_i},{x_j})>{r_j} \geqslant {r_i},$ then $\dfrac{r_i^2 + r_j^2 -{\rho }^2({x_i},{x_j})}{2{r_i}{r_j}}<\dfrac{1}{2}$  and  \[r_i^2 + r_j^2-{r_i}{r_j}<{{M}^2(A)} \cdot [r_i^2 + r_j^2 - 2{r_i}{r_j}cos{\theta _y}({x_i},{x_j})].\]
	
	Thus \[cos{\theta _y}({x_i},{x_j})<1-\dfrac{1}{2{M}^2(A)}.\]
	
	\!\!\!\!\!\!\!\!\!\!\it \textbf{\ \ \ \ \ Claim:} If ${x_i}'$ denotes the intersection point of $\partial B(y,r)$ with $r=\dfrac{r_A}{16}$ and the minimal geodesic between $x_i$ and $y$, we can find a constant $\varepsilon(A)$ such that ${cos\theta _y}({x_i},{x_j})=cos\theta _y({x_i'},{x_j'})\geqslant 1-\dfrac{1}{2{M}^2(A)}$ for any $\{i,j\}$ satisfying ${\rho }({x_i'},{x_j'})\leqslant \varepsilon(A)r$.\\\\ 
	\!\!\!\!\!\!\!\!proof of claim: \rm Set \[{A^{\frac{{{{{r}}_A}}}{2}}} = \{ x:\exists\  {x_A} \in A\ \ such\ that\ \ \rho (x,{x_A}) \leqslant \frac{{{{{r}}_A}}}{2}\}. \]
	Since $\partial B(y,r)\in {A^{\frac{{{{{r}}_A}}}{2}}}$, there exists a largest number $\varepsilon(A,y,s)$ for each $s \in \partial B(y,r)$ such that \[{cos\theta _y}(s,t)\geqslant 1-\dfrac{1}{2{M}^2(A)}\] for any point $t \in B(s,\varepsilon(A,y,s)r)$.
	
	Since $\varepsilon(A,y,s)$ is continues with rsepect to $s\in \partial B(y,r)$ and $\partial B(y,r)$ is compact, we can find the number\[\varepsilon (A,y) = min\{ \varepsilon (A,y,s):s \in \partial B(y,r)\}\]such that for any $s \in \partial B(y,r)$, ${\theta _y}(s,t)\geqslant 1-\dfrac{1}{2{M}^2(A)}$ if $t \in B(s,\varepsilon(A,y)r).$ 
	
	Since $\varepsilon(A,y)$ is continues with repect to $y\in {A^{\frac{{{{{r}}_A}}}{2}}}$ and $\overline {A^{\frac{{{{{r}}_A}}}{2}}}$ is compact, then we can find a $\varepsilon(A)$ such that for any $y\in {A^{\frac{{{{{r}}_A}}}{2}}}$, \[{cos\theta _y}(s,t)\geqslant 1-\dfrac{1}{2{M}^2(A)}\] if $\{s,t\} \subset \partial B(y,r)$ and $\rho(s,t)\leqslant\varepsilon(A)r.$\\
	
	This completes the proof of Claim. \\
	 
	Thus $\{x_i'\} _{i = 1}^k$ is a strict $\varepsilon(A)r$-net in $\partial B(y,r)$ and  then $k\leqslant C_1(A)$ while $C_1(A)$ denotes the largest cardinality of strict $\dfrac{\varepsilon(A)r_A}{16}$-net in $\partial B(a,\dfrac{r_A}{16})$ for any $a\in A$.\\
\end{proof}

\begin{definition}
    \rm	$(X,d)$	is a metric space, we say $d$ is $finite$ $dimensional$ on a subset $Y \subset X$ if there exist $C(Y) \in [1, \infty)$ and $R(Y) \in (0,+\infty ]$ such that $Card\;\cal B \leqslant $  $C(Y)$ for every Besicovitch family $\cal B$ $=\{ B({x},{r}):{x} \in Y, {r} \leqslant R(Y)\}$ in $(X,d)$.\\
\end{definition}

\begin{definition}
	\rm $(X,d)$	is a metric space, we say $d$ is $\sigma$-\it{finite}\rm $\ dimensional$ on $X$ if $X$ can be written as a countable union of subsets on which $d$ is finite dimensional.\\
\end{definition}

\begin{proposition} $[23]$
    \rm Let $(X,d)$	be a complete saperable metric space. The differentiation theorem holds for every locally finite Borel regular measure over $(X,d)$ if and only if $d$ is $\sigma$-finite dimensional.\\
\end{proposition}

From Theorem 2 , $\rho$ is $\sigma$-finite dimensional in a complete Riemannian manifold $(M,\rho)$.
 Thus we have 
\begin{corollary}
	\rm Suppose $(M,\rho) $ is a complete Riemannian manifold, 
	\[\mathop {\lim }\limits_{{\text{r}} \to 0} \frac{1}{{\mu (B(x,r))}}\int\limits_{B(x,r)} {f(y)d\mu (y) = f(x)} \quad \quad     \mu  - a.e.\]
	for $\forall$ Radon measure $\mu$.\\
\end{corollary}

\subsection{Besicovitch-Morse type theorems in Riemannian manifold}

\begin{definition}($\tau $-satellite configuration)\rm {[6]}
	$(X,d)$	\rm{is a metric space. Fix} $\tau >\text{1 }$ and ${I_n} = \{ 1,2, \cdots ,n\}$. Let $\{{a_i}:i\in{I_n}\}$
	and  $\{{S_i}:i\in{I_n}\}$ be, respectively, an ordered set of points and an ordered set of bounded subsets in $X$. We say that $\{{S_i}:i\in{I_n}\}$ is in $\tau$ -$satellite$ $configuration$ with respect to $\{{a_i}:i\in{I_n}\}$ if the following conditions hold for each $i\in{I_n}$ and some index ${i_0}\in{I_n}$ called the central index:
	
	$1$): ${a_i}\in{S_i}$,
	
	$2$): ${S_{{i_o}}} \cap {S_i} \ne \emptyset$,
	
	$3$): $diam{S_{{i_o}}} < \tau  \cdot diam{S_i}$,
	
	$4$): if $i < j \leqslant n$, then ${a_j} \notin {S_i}$ and $diam{S_j} < \tau  \cdot diam{S_i}$.\\
\end{definition}

\begin{proposition}
	\rm$(M,\rho)$ is a complete Riemannian manifold, $A$ is an arbitrary bouned subset of $M$ and $\;{r_A}$ denotes $\; \frac{1}{4}\mathop {\inf }\limits_{x \in A} in{j_x}M$.
	
	Suppose we are given a set ${I_n} = \{ 1,2, \cdots ,n\}$ and an ordered set $\{{S_i}\subset M:i\in{I_n}\}$ of bounded sets in $\tau$-satellite configuration with respect to an ordered set $\{{a_i}:i\in{I_n}\}\subset A$.\\
	
	Also suppose that a ball $B({a_i},{r_i})$ with ${r_i} < \frac{{{r_A}}}{{4\lambda  + \frac{1}{{6\lambda }}}}$ is contained in ${S_i}$ for each $i\in{I_n}$. 
	
	For the central index  $i_0$, the minimal geodesic between ${a_{{i_o}}}$ and ${a_i}$ and  $\partial B({a_{{i_o}}},\frac{{{r_A}}}{16})$ intersect at 
	${a_{\text{i}}}^\prime$. \\
	
	Assume that there are constants $C>0$ and $D \geqslant$ 1 such that for any pair of distinct indexes $i$ , $j$ satisfying
	
	$1$): $D{r_{{i_o}}}<\rho({a_{{i_o}}},{{a_i}})<\rho({a_{{i_o}}},{{a_j}})$,
	
	$2$): $\rho({a_{\text{i}}}^\prime,{a_{\text{j}}}^\prime)\leqslant C$,\\
	
	\!\!\!\!\!\!\!\!we all have $\{{a_i}\in{S_j}\}$ if ${r_{{k}}} < \frac{{{r_A}}}{{D + \frac{1}{{6\lambda }}}}$ for each $k\in\{1,2,\cdots,n\}$.\\
	
	If we set $\lambda  = \mathop {\max }\limits_{i \in {I_n}} \frac{{diam{S_i}}}{{2{r_i}}}$, then there exists a $N(A,\lambda,C,D)$ related to $A,\lambda,C$ and $D$ such that $Card({I_n})\leqslant N(A,\lambda,C,D)$.
	
\end{proposition}

\begin{proof}[\textbf{\textit{Proof}}]
	From the definition of $\tau $-satellite configuration, \[\rho ({a_i},{a_j}) > {r_i} \geqslant \dfrac{{diam{S_i}}}{{2\lambda }} \geqslant \dfrac{{diamS}}{{2\lambda \tau }} \geqslant \dfrac{r}{{2\lambda }}\ \ for\ \forall \ i < j.\] 
	Thus $Card\{ i \in {I_n},\rho (a,{a_i}) \leqslant {D}r\}$ can not exceed $N(A,D,\lambda)$ which denotes the largest cardinality of $\dfrac{r}{2\lambda }$-net in $B(a,Dr)$ for each $a\in A$ and $r\leqslant min\{ \dfrac{{{r_A}}}{{4\lambda+ \frac{1}{{6\lambda }}}},\dfrac{{{r_A}}}{{D+ \frac{1}{{6\lambda }}}}\}$.\\
	
	\!\!\!\!\!\!\!\!\!\!\it \textbf{Claim:} There exists a $N'(A,D,\lambda)$ such that $a_i$ belongs to at most $N'(A,D,\lambda)$ elements of ${\{ B({x_j},{r_j})\} _{j \in {I_n}}}$ for each ${i\in {I_n}}$.
	\\\\
	\!\!\!\!\!\!\!\!proof of claim: \rm If $a_i\in S_j$, then \[\rho ({a_i},{a_j}) < diam{S_j} < \tau  \cdot diam{S_i} \leqslant 2diam{S_i} \leqslant 4\lambda {r_i}.\]
	Thus $a_j\in B({x_i},4{r_i})$ and $\rho ({a_i},{a_j})\geqslant r_i\geqslant\dfrac{r_i}{2\lambda}$ since $\lambda\geqslant\dfrac{1}{2}.$
	
	If $a_i\in S_k$ and $k>j$, then\[\rho ({a_k},{a_j}) \geqslant r_j\geqslant \dfrac{diamS_j}{2\lambda} \geqslant \dfrac{\rho ({a_i},{a_j})}{2\lambda} \geqslant \dfrac{r_i}{4\lambda}.\]
	
	Thus $Card\{ j \in {I_n}: a_i\in S_j\}$ can not exceed $N'(A,D,\lambda)$ which denotes the largest number of $\dfrac{r}{2\lambda }$-net in $B(a,4\lambda r)$ for each $a\in A$ and $r\leqslant min\{ \dfrac{{{r_A}}}{{4\lambda+ \frac{1}{{6\lambda }}}},\dfrac{{{r_A}}}{{D+ \frac{1}{{6\lambda }}}}\}$.
	
	This completes the proof of Claim. \\
	
	Set \[I' = \{ i \in {I_n}:\rho (a,{a_j}) > Dr\},\]\[\mathcal {F}=\{B(a_i,r_i):i\in I'\}.\] 
	
	Select an arbitrary ball $B_1=B(a_{i_1},r_{i_1})$ such that $a_{i_1}\in I'$ and set \[\mathcal{F}_1 = \mathcal {F} - \bigcup\limits_{{a_{{i_1}}} \in B({a_i},{r_i})}^{} {B({a_i},{r_i})} .\] 
	
	Inductively, select an arbitrary ball $B_k=B(a_{i_k},r_{i_k})$ such that $B(a_{i_k},r_{i_k})\in \mathcal{F}_{k-1}$ and set \[\mathcal{F}_k = \mathcal {F}_{k-1} - \bigcup\limits_{{a_{{i_{k-1}}}} \in B({a_i},{r_i})}^{} {B({a_i},{r_i})} \]while  $\bigcup\limits_{{a_{{i_k}}} \in B({a_i},{r_i})}^{} {B({a_i},{r_i})}$ includes at most $N'(A,D,\lambda)$ elements for each $k$ by the claim.
	
	Suppose we have selected $\{ {B_1},{B_2}, \cdots ,{B_k}\}$, then \[{a_{{i_s}}} \notin {B_t}\ and\  {a_{{i_t}}} \notin {B_s}\ for\ \forall\  1\leqslant s \ne t\leqslant k.\] 
	By the assumed condition, \[\rho ({a_{{i_s}}'},{a_{{i_t}}'})>C\  for\ \forall\  1\leqslant s \ne t\leqslant k\] and then k can not exceed  $N(A,C)$ which denotes the largest cardinality of strict $C$-net in $\partial B(a,\dfrac{r_A}{16}).$\\
	
	At last, \[N(A,\lambda ,C,D) = N(A,D,\lambda) + N'(A,D,\lambda ) \cdot N(A,C).\] 
	
\end{proof}

Next we introduce a constant $ \alpha (A)$ related to a bounded set $A$ in a complete Riemannian manifold $(M,\rho)$. By the proof of Bishop Comparison theorem, there are two constants $M_1(A)$ and $M_2(A)$ such that\[{M_1}(A) \leqslant \left| {Jacbian\ of\ {{\exp }_a}} \right| \leqslant {M_2}(A)\]
for each $a\in A$ and then we set $ \alpha (A)=\dfrac{{M_2}(A)}{{M_1}(A)}.$\\

\begin{proposition}
	$(M,\rho)$ \rm is a complete Riemannian manifold and $A$ is a bounded subset of $M$. Fix $\text{1 }<\tau  \leqslant 2$ and $\lambda  \geqslant 1$, then there exists $C(A,\lambda) \in \mathbb{N}$ such that
	 if a finite ordered collection of convex sets $\{{S_i}:i = 1,2, \cdots ,n\}$ whose elements are subsets of $M$ and satisfy $diamS_i<min\{\dfrac{r_A}{8\alpha(A)},\dfrac{r_A}{8}\}$ is a $\tau $-satellite configuration with respect to an ordered set $\{{a_i}:i = 1,2, \cdots ,n\}$ while there exists ${r_i}<\dfrac{{{r_A}}}{{32{\lambda ^2} \cdot \alpha (A)+ \frac{1}{{6\lambda }}}}$ for each ${a_i}$ such that 
	 \[B({a_i},{r_i}) \subset S_i(a) \subset B({a_i},\lambda{r_i}),\]
	 then $n$ can not exceed $C(A,\lambda)$	 
\end{proposition}

\begin{proof}[\textbf{\textit{Proof}}] 

	Here we will apply the  proposition 4. $S$ and $B(a,r)$ respectively denotes the set and the ball associated with the central index  $i_0$, and the minimal geodesic between $a$ and ${a_i}$ intersects  $\partial B({a},\dfrac{{{r_A}}}{16})$  at 
	${a_{\text{i}}}^\prime$.\\
	
	Select distinct $b,c\in \{a_1,a_2,\dots,a_n\}$ satisfying\[\rho (a,c) \geqslant \rho (a,b) > 32{\lambda ^2} \cdot \alpha (A) \cdot r\ \ and\ \ \rho (b',c') \leqslant \frac{{\frac{{{r_A}}}{{16}}}}{{16\lambda  \cdot \alpha (A)}}.\]
	Suppose $c$ denotes $a_j$. If we can deduce that $b\in S_j$, then \[C(A,\lambda)=N(A,\lambda,\frac{{\frac{{{r_A}}}{{16}}}}{{16\lambda  \cdot \alpha (A)}}),32{\lambda ^2} \cdot \alpha (A)\]by Proposition 4.\\
	
	Consider the exponential map $exp_a$, draw a straight line from ${exp_a}^{-1}(a)$ to ${exp_a}^{-1}(b)$ in $T_a(M)$ and extend it to $T$ such that $\rho (a,exp_a(T)) = \rho (a,c).$ Then select a point $x\in S \cap S_i$, draw a straight line from ${exp_a}^{-1}(x)$ to ${exp_a}^{-1}(b)$ in $T_a(M)$ and extend it to $y$ such that \[d({exp_a}^{-1}(x),y)=\dfrac{{\rho (a,c)}}{{\rho (a,b)}}\cdot d({exp_a}^{-1}(x),{exp_a}^{-1}(b))\] while $d$ denotes the Euclid distance in $T_a(M)$. \\
	
	\begin{figure}[h]
		\centering
	
	\end{figure}
	
	\!\!\!\!\!\!\!\!\!\!\it \textbf{Claim:} $exp_a(y)\in B(c,r_j)\subset S_j$. 
	
	\!\!\!\!\!\!\!\!proof of claim: \rm \[\rho ({\exp _a}(y),c) \leqslant \rho ({\exp _a}(y),{\exp _a}(T)) + \rho ({\exp _a}(T),c).\]
	
	Since \[\rho ({\exp _a}(T),a) = \rho (a,c) < diamS + diam{S_i} <  {r_A},\]
    \[\rho (a,x) < diamS < {r_A}\ \  and\ \ \rho (a,{\exp _a}(y)) < \rho (a,x) + s\cdot\alpha(A) \rho (b,x) < {r_A},\] then
    \[\rho (a,x) \geqslant {M_1}(A) \cdot d(0,{\exp _a}^{ - 1}(x))\]and \[\rho ({\exp _a}(y),{\exp _a}(T)) \leqslant {M_2}(A) \cdot d(y,T) = {M_2}(A) \cdot (s - 1)d(0,{\exp _a}^{ - 1}(x)).\]
    If $s$ denotes $\dfrac{{\rho (a,c)}}{{\rho (a,b)}}$, then \[\rho ({\exp _a}(y),{\exp _a}(T)) \leqslant (s - 1)\frac{{{M_2}(A)}}{{{M_1}(A)}}\rho (a,x) = (s - 1) \cdot \alpha (A)\rho (a,x)\]
    by the Bishop Comparison theorem.\\
	 
	Next,\[
	\rho ({\exp _a}(T),c) \leqslant {M_2}(A) \cdot d(T,{\exp _a}^{ - 1}(c)) \hfill \]
	and \[\rho (b',c') \geqslant {M_1}(A) \cdot d({\exp _a}^{ - 1}(b'),{\exp _a}^{ - 1}(c')). \]
	Since \[\frac{{d(T,{{\exp }_a}^{ - 1}(c))}}{{d({{\exp }_a}^{ - 1}(b'),{{\exp }_a}^{ - 1}(c'))}} = \frac{{\rho (a,c)}}{{\frac{{{r_A}}}{{16}}}},\]then \[\rho ({\exp _a}(T),c) \leqslant \frac{{{M_2}(A)}}{{{M_1}(A)}}\frac{{\rho (a,c)}}{{\frac{{{r_A}}}{{16}}}}\rho (b',c') = \alpha (A) \cdot \rho (a,c) \cdot \frac{{\rho (b',c')}}{{\frac{{{r_A}}}{{16}}}} \leqslant \frac{{\rho (a,c)}}{{16\lambda }}.\]
	Besides,\[\rho (a,b) > 32{\lambda ^2} \cdot \alpha (A) \cdot r \geqslant 16\lambda  \cdot \alpha (A) \cdot diamS\]
	and then\[\rho ({\exp _a}(y),{\exp _a}(T)) \leqslant s\cdot \alpha (A)\rho (a,x) \leqslant s \cdot \alpha (A)\cdot diamS\leqslant\frac{{\rho (a,c)}}{{16\lambda }}.\]
	Now \[\rho ({\exp _a}(y),c) \leqslant \rho ({\exp _a}(y),{\exp _a}(T)) + \rho ({\exp _a}(T),c)\leqslant\frac{{\rho (a,c)}}{{8\lambda }}, \]and then
	
	\[\begin{gathered}
	\frac{{\rho (a,c)}}{{8\lambda }} \leqslant \frac{{\rho (a,x) + \rho (c,x)}}{{8\lambda }} \leqslant \frac{{diamS + diam{S_j}}}{{8\lambda }} \leqslant \frac{{(\tau  + 1)diam{S_j}}}{{8\lambda }} \hfill \\
	\ \ \ \ \ \ \ \ \ \ \leqslant \frac{{3diam{S_j}}}{{8\lambda }} \leqslant \frac{{diam{S_j}}}{{2\lambda }} \leqslant {r_j}. \hfill \\ 
	\end{gathered} \]
	This completes the proof of Claim. \\
	
	Since $b$ lies in the minimal geodesic between $x$ and $exp_a(y)$ and $S_j$ is convex, then $b\in S_j$ because $x\in S_j$. 
\end{proof}

According to Theorem $5.4$ of [6] and Proposition 5, we have

\begin{theorem}\rm(Besicovitch-Morse type theorem)
 $(M,\rho)$ \rm is a complete Riemannian manifold, $A$ is an arbitrary bounded subset of $M$ and $\;{r_A}$ denotes $\; \frac{1}{4}\mathop {\inf }\limits_{x \in A} in{j_x}M$. For each point $a \in A$, associate $a$ with a convex set $S(a)$ containing $a$ such that $diamS<min\{\dfrac{r_A}{8\alpha(A)},\dfrac{r_A}{8}\}$ and there exists a $r < \frac{{{r_A}}}{{4\lambda  + 1}}$  such that \[B(a,r) \subset S(a) \subset B(a,\lambda r).\]
 
 Then for some $N\in \mathbb{N}$ not exceeding $N(A,\lambda)$, there are subsets $A_1,A_2,\cdots, A_N$ of $A$ such that \[A \subseteq \bigcup\limits_{i = 1}^N {\mathop  \cup \limits_{a \in {A_i}} } S(a)\] and the elements of the collection $\{ S(a):a \in {A_i}\}$ are pairwise disjoint for each $i$.\\
\end{theorem}

\section{Relationship between BCP and WBCP}
 \qquad In this section,we recall some known facts of the relationship between BCP and WBCP in metric  spaces and then give a sufficient condition for the equivalence of BCP and WBCP.\\
 
\subsection{BCP $\Rightarrow$ WBCP}

\qquad If a metric space $(X,d)$ satisfies BCP, then $(X,d)$ satisfies WBCP with the same constant. This fact is not difficult to obtain and you can refer to [1] Remark 3.3 for a simple proof.\\

WBCP is strictly weaker than BCP. [9] Example 3.4 has shown the existence of a metric space with WBCP not satisfying BCP. \\

\subsection{A sufficient condition for WBCP $ \Rightarrow$ BCP }

\begin{definition}$[9]$ \rm(doubling metric space)	A metric space $(X,d)$ is said to be doubling if there is a constant $C \geqslant 1$ such that for $r>0$, each ball in $(X,d)$ with radius $2r$ can be covered by a family of at most $C$ balls of radius $r$.\\
\end{definition}

\begin{proposition}$[9]$
	\rm Let $(X,d)$	be a doubling metric space. Then $(X,d)$ satisfies BCP if and only if $(X,d)$ satisfies WBCP.
\end{proposition}

    Here we give a different class of metric spaces which satisfy BCP if and only if they satisfy WBCP. Next we give a property that is named here for the first time as far as I know.\\

\begin{definition}\rm(contraction intersecting property)
	 $(X,d)$	is a metric space and $m$ is a positive integer, we say that $(X,d)$ has the \it contraction intersecting property \rm if and only if there exists a constant $\beta\in (0,1)$ such that for any collection of closed balls ${\cal F}$ ${ = \{ B(x,r):\beta R \le r \le R\}}$ with $R>0$, we can find two constants $s(m)\in (0,1)$ and $C(m)\in\mathbb{N}$ such that 
	if \[\mathop  \cap \limits_{i = 1}^{C(m) + 1} B{({{{x}}_i},{{r}}_i)} \ne \emptyset, \]  there exists $m+1$ indexes $\{ {j_1},{j_2}, \cdots ,{j_{m + 1}}\}$ of $\{ 1,2, \cdots ,C(m)+1\}$ such that 
	
	\[\mathop  \cap \limits_{i = 1}^{m+ 1} B{({{{x}}_{j_i}},{s(m)\cdot{r}}_{j_i})} \ne \emptyset. \]\\
\end{definition}

\begin{example}
	\rm For  $({\mathbb{R}^n},\left\|  \cdot  \right\|)$ with \[\left\| {({x_1},{x_2}, \cdots ,{x_n})} \right\| = \sqrt{\sum\limits_{i = 1}^n {x_i^2} }, \] we can prove that $C(m)=2m$.
\end{example}
$ Proof: $ We set $n=2$ for simplification.\\

Suppose that every ball of ${\cal F}$ shares the same radius because $\beta$ can be closed to 1.\\

Now we consider a limit condition such that\[\mathop  \cap \limits_{i = 1}^{2m + 1} B{({{{x}}_i},{{r}}_i)} = {\left\{{\left.0 \right\}} \right.}\]  and the original point${\left\{{\left.0 \right\}} \right.}$ is exactly on the boundary of each ball.\\

 Thus we get $2m+1$ tangent lines at the original point. Since there are finite lines, we can obtain them one by one. Before we obtain the first line, assign 0 to the whole plane. Once we obtain a tangent line, add one to the half-plane where the corresponding ball lies. After we set all lines, the whole plane are divided to $4m+2$ parts (maybe some parts are $\emptyset$) and each of them has a value of ${{\{0,1,2,}} \cdots {{,2m + 1\} }}$.\\
 
 At each step, we add one to a series of connected $2m+1$ parts. Then every part and the part which is $2m$ parts apart from it can not be added one at the same step, thus the sum of the values of these two parts is exactly $2m+1$ and one of them has the value larger than $m$. We note that they shares the same angle and then half of the plane has a value larger than $m$. Since the half of the plane consists of $2m+1$ parts, there must be one\\\\ part with an angle larger than  $\ \dfrac{\pi }{{2m + 1}}$.\\

 Actually if the value of one of the parts is larger than $m$, the small area which is around the original point and lies in this part belongs to at least $m+1$ balls. We choose a part with an angle larger than $\ \dfrac{\pi }{{2m + 1}}$ and the value of it is larger than $m$. In this part, we can find an area which still belongs to at least $m+1$ balls after we reduce the radius of all balls.\\
 
If $n > 2$, we can use the same method. Just replace $4m+2$ with $2k$ and the half-space consists of k parts.\\

\begin{theorem}
	\rm $(X,d)$ is a metric spaces with the contraction intersecting property, if there exists a Radon measure $\mu$ such that \[\mathop {\inf }\limits_{x \in X} \mu (B(x,r)) = \sigma (r) > 0\] for $\forall r > 0$, then BCP and WBCP are equivalent in $(X,d)$. 	
\end{theorem}

\begin{proof}[\textbf{\textit{Proof}}]
	We just need to prove that $(X,d)$ satisfies the BCP if it satisfies the WBCP with a constant $C_1$.
	Given $R,\ \beta$ as in the Definition 10 and A is a bounded set in X. 
	\[{\cal F}= \{ B(x,r):x \in A,0 < r \le R\} \]
	and each point of $A$ is the center of some ball of $\mathcal{F}$.
	\[{D^n} = \{ B(x,r)\in\mathcal {F},{\beta ^n}R \leqslant r \leqslant {\beta ^{n - 1}}R\} \]
	and ${E^{\text{n}}}$ denotes the centers of the balls in ${D^n}$.
	Thus ${\cal F}$=$\bigcup\limits_{n = 1}^\infty  {{D^n}} $ and $A = \bigcup\limits_{n = 1}^\infty  {{E^n}}$.
	
	Set 
	\[{\mathcal{F}^{\text{1}}}{\text{ = }}{D^{\text{1}}},\quad {A_{\text{1}}}{\text{ = }}{E^{\text{1}}},\]
	\[{A_{\text{2}}}{\text{ = }}{E^{\text{2}}} - \bigcup\limits_{B \in {\mathcal{F}^{\text{1}}}}^{} B ,\quad {\mathcal{F}^{\text{2}}}{\text{ =}} \{ B(x,r)\in\mathcal{F} :x \in {A_{\text{2}}}\},\]
	inductively, \[{A_{{k}}} = {E^k} - \bigcup\limits_{i = 1}^{k - 1} {\bigcup\limits_{B \in {\mathcal{F}^{{i}}}}}B,\quad {\mathcal{F}^{{k}}}{\text{ =}} \{ B(x,r)\in\mathcal{F} :x \in {A_{\text{k}}}\}.\]
	
    Then we can deduce that if a pair of balls are respectively from distinct $\mathcal{F}^{{i}}$ and $\mathcal{F}^{{j}}$ , the center of each of them does not belong to the other. 
    
    Thus for an arbitrary point $y\in X$, $\mathcal{F}_{y}$ denots all balls from $\mathcal{F}$ which contains $y$ and there are at most $C_{1}$ indexes  $\{ {y_1},{y_{2,}} \cdots ,{y_n}\}$  such that \[{{\cal F}^{{y_i}}} \cap {{\cal F}_y} \ne \emptyset .\]
    
    \!\!\!\!\!\!\!\!\it \textbf{Claim 1:} For each $i$, there exists a collection of balls $\{ {B^i_j} =B({x^i_j},{r^i_j})\} _{j = 1}^{{k(i)}}$ such that 
    ${A_{i}}$ is covered by
    \\\\ the balls of $\{ {B_j(i)}\} _{j = 1}^{{k(i)}}$ and $y$ belongs to at most $C_{2}$ balls of $\{ {B_j(i)}\} _{j = 1}^{{k(i)}}$.\\
    
    \!\!\!\!\!\!\!\!proof of claim 1:   \rm Fix    $0<\alpha<1$, $C_{2}>0$, $0<s<1$ and $0<\beta<1$. ( Here $s=s(m)$ is as in the definition 10 for $m=C_{1}$ and $C_{2}$ denotes $C(m)$. )
    \[{\mathcal{F}^{i}}= \{ B(x,r):x \in A_{i} , {\beta ^i}R \leqslant r \leqslant {\beta ^{i - 1}}R\}\]
    We choose \[B^i_1=B({x^i_1},{r^i_1})\in {\mathcal{F}^{i}} \quad with\quad r^i_1>sR,\] 
    inductively,
    \[B^i_{k}=B({x^i_k},{r^i_k})\in {\mathcal{F}^{i}} \quad with\quad {x^i_k}\in {A_i} - \bigcup\limits_{j = 1}^{k - 1}{{B^i_j}} \quad and\]
    \[{r^i_k} > s \cdot sup\{ r(x):B(x,r(x)) \in {\mathcal{F}^i}\ with\ x \in {A_i} - \bigcup\limits_{j = 1}^{k - 1} {{B^i_j}} \}.\]
    
    \!\!\!\!\!\!\!\!\it \textbf{Claim 2:} There exists a ${k_i}$ such that ${A_i}$ is covered by $\bigcup\limits_{j = 1}^{k_{i}} {{B^i_j}}.$
    
    \!\!\!\!\!\!\!\!proof of claim 2:     \rm Since $A_i$ is bounded and $R$ is finite, \[\mu (\bigcup\limits_{B \in {\mathcal{F}_i}}^{} {B) \  is\  finite} \quad for\quad \forall\  Radon\ measure\   \mu \  on\  (X,d).\]
    
    For $\forall$ $j>t$, we have 
    $r_{t}>sr_{j}\ and\ \rho({x_t},{x_j})>r_{t}>sr_{j}.$ Then \[\rho({x_t},{x_j})>s\ max\{r_{t},r_{j}\}>\dfrac{s}{2}{r_{t}}+\dfrac{s}{2}{r_{j}} . \]
    
    Thus the balls in $\{ B({x_t},\dfrac{s}{2}{\beta ^i}R)\} _{t = 1}^{k_i}$ are pairwise disjoint and then \[k_i \leqslant \dfrac{\mu (\bigcup\limits_{B \in {\mathcal{F}_i}}^{} {B})}{\mathop {\inf }\limits_{x \in X} \mu (B(x,\dfrac{s}{2}{\beta ^i}))}\mathop . \]
    
    This completes the proof of Claim 2.
  \\\\
  \textit{continuation of the proof of Claim 1:}
  \rm Since $\rho({x_t},{x_j})>s\ max\{r_{t},r_{j}\}$, then $\{ B({x_t},s{r_t})\} _{t = 1}^{k_i}$ is a Besicovitch family of $(X,d)$. Thus \[\sum\limits_{t = 1}^{{k_i}} {{\chi _{B({x_t},s{r_t})}} \leqslant {C_1}=m}. \]
  
  If for an arbitrary family $\{ B({x_i},{r_i})\} _{i = 1}^{k}\subset \mathcal{F}\  cover\  A_{i},$ there always exists a $y$ such that
  \[\sum\limits_{i = 1}^{k} {{\chi _{B({x_i},{r_i})}}(y) > C_2=C(m)}.\]
  
  Then for the family $B({x^i_j},{r^i_j})\} _{j = 1}^{k_i}\subset \mathcal{F}\  cover\  A_{i}$, there exists a $y$ such that
  \[\sum\limits_{i = 1}^{k_i} {{\chi _{B({x^i_j},{r^i_j})}}(y) > C(m)}.\]
  
  Thus from the definition of the contraction intersecting property we can find a $z$ such that\[\sum\limits_{j = 1}^{k_i} {{\chi _{B({x^i_j},s{r^i_j})}}(z) > m},\]
  which deduces the contradiction.\\

  Then there exists a family $B({x_j(i)},{r_j(i)})\} _{j = 1}^{k(i)}\subset \mathcal{F}\  cover\  A_{i}$ such that \[\sum\limits_{i = 1}^{k(i)} {{\chi _{B({x_j(i)},{r_j(i)})}} \leqslant C_2}.\]This completes the proof of Claim 1.
  \\\\
  \textit{continuation of the proof of Theorem 4 :}
  From Claim 1, we can find a family $\{ {B_j(i)} =B({x_j(i)},{r_j(i)})\} _{j = 1}^{{k(i)}}$ such that 
  ${A_{i}}$ is covered by the balls of $\{ {B_j(i)}\} _{j = 1}^{{k(i)}}$ and an arbitrary $y \in X$ belongs to at most $C_{2}$ balls of $\{ {B_j(i)}\} _{j = 1}^{{k(i)}}$ for each $i$. ${\cal F} (i)$ denotes $\{ {B_j(i)} \} _{j = 1}^{{k(i)}}$ from claim 1 for each $i$. We have known that there are at most $C_{1}$ indexes  $\{ {y_1},{y_{2,}} \cdots ,{y_n}\}$  such that \[{{\cal F}^{{y_i}}} \cap {{\cal F}_y} \ne \emptyset.\] Thus if $\mathcal F'$ denotes $\bigcup\limits_{i = 1}^\infty  {{{\cal F}_i}} $, then $\mathcal F'$ is a subcollection of $\mathcal F$ and satisfies \[{\chi _A} \leqslant \sum\limits_{B \in \mathcal{F}'} {{\chi _B}}  \leqslant {C_1}{C_2}.\] 
\end{proof} 
 
\section{The best constants}
In $({\mathbb{R}^n},\left\|  \cdot  \right\|)$ with \[\left\| {({x_1},{x_2}, \cdots ,{x_n})} \right\| = \sqrt {\sum\limits_{i = 1}^n {x_i^2} }, \] the best constant of Besicovitch covering theorem is difficult to obtain when $n>1$. And the investigations of this question is interesting when they are related to geometric problems. Refer to [2],[7],[12] and [25] for the intersting investigations.\\

\begin{notation}
	\rm Here are some symbols in $({\mathbb{R}^n},\left\|  \cdot  \right\|)$ to be  defined.\\\\
    $w(n)$ denotes the best constant of the weak Besicovitch covering property(WBCP) in $({\mathbb{R}^n},\left\|  \cdot  \right\|)$. \\\\
    $H^*(n)$: The strict Hadwiger number $H^*(n)$ is the largest number of (closed) unit balls in $({\mathbb{R}^n},\left\|  \cdot  \right\|)$ that can touch the (closed) unit ball with center in the original point and these unit balls are disjoint.\\\\
	$K(n)$ denotes the largest number of balls that can form a configuration in $({\mathbb{R}^n},\left\|  \cdot  \right\|)$ while each ball of the configuration intersects every other ball but does not contain the center of any other ball.\\\\
	$\alpha(n)$ denotes the best constant of Besicovitch covering theorem in $({\mathbb{R}^n},\left\|  \cdot  \right\|)$. \\\\
	$\beta(n)$ denotes the largest number of (open) unit balls in $({\mathbb{R}^n},\left\|  \cdot  \right\|)$ that can be packed into a ball of radius $5$ with the requirement that one of their centers is the center the ball of radius $5$.
	
\end{notation}

The references have showed that \[w(n) \leqslant K(n)\leqslant\alpha(n)\leqslant\beta(n)\leqslant5^n.\]

Here are the best ranges of the above numbers so far.   \\

\begin{tabular}{|c|c|c|c|c|c|}
	\hline
	&$n=1$   &$n=2$  &$n=3$  &$n=4$  & ${\text{n}} \to \infty $ \\ 
	\hline 
	$w(n)=H^*(n)$& $2$  & $5$ &$12$  & $24$ & $[({1.1547+o(1))}^n,1.3205^{(1+o(1)n}]$ \\ 
	\hline 
	$K(n)$&$2$  &$[8,11]$  &  &  &$[{1.25}^n,{(1.887+o(1))^n}]$  \\ 
	\hline 
	$\alpha(n)$& 2 & $[8,19]$ &$[12,87]$  & $[24,331]$ &$[{1.25}^n,{(2.641+o(1))^n}]$   \\ 
	\hline 
	$\beta(n)$& 5 & 19 & $[67,87]$ &  $[226,331]$&$[{(2.065+o(1))^n},{(2.641+o(1))^n}]$  \\ 
	\hline 
\end{tabular} \\\\

Although someone has mentioned that the best constant of the Besicovitch covering theorem in $({\mathbb{R}^1},\left\|  \cdot  \right\|)$ is 2, I have never seen an exact proof of it when the set $A$ in the thoerem is an arbitrary unbounded set. Now we give a proof of it here. \\

\begin{theorem} \rm(Besicovitch covering theorem in $({\mathbb{R}^1},\left\|  \cdot  \right\|)$)
	\rm 	$A$ is an arbitrary subset of ${\mathbb{R}^1}$ and $R$ is a positive real number. For every collection of closed intervals \[{\cal F}=\{[x-r,x+r]: x\in A, r \leqslant R\}\] such that each point of $A$ is the center of some interval of $\mathcal{F}$,
	we can find 2 subcollections  ${\mathcal{F}_1}$ and ${\mathcal{F}_2}$ of disjoint intervals such that\[A \subset \bigcup\limits_{i = 1,2}^{} {\bigcup\limits_{I \in {\mathcal{F}_i}} I }.\]
\end{theorem}

\begin{proof}
	\rm
	\textbf{Conditon 1:} $A$ is compact.
	
	Then we can choose finite subcollections $\left\{ {{I_i}} \right\}_{i = 1}^N$ such that $A\subset \bigcup\limits_{i = 1}^N {{I_i}}$ and it is not difficult to choose two subcollections ${\mathcal{F}_1}$ and ${\mathcal{F}_2}$  such that $\left\{ {{I_i}} \right\}_{i = 1}^N={\mathcal{F}_1} \cup {\mathcal{F}_2}$ while each ${\mathcal{F}_i}$ consists of disjoint intervals. Moreover, if [ ] denotes an element in ${\mathcal{F}_1}$ and ( ) denotes \\an element in ${\mathcal{F}_2}$, then the chosen ${\mathcal{F}_1} \cup {\mathcal{F}_2}$ consists of finite chains with a same pattern of ( [ ) ( ] [ ) ( ] [ ) ].\\
	
	\textbf{Conditon 2:} $A$ is bounded.\\
	
	Choose \[I_1=[x_1-r_1,x_1+r_1], \ \ x_1\in A,\ \ r_1>\dfrac{1}{2}sup\{r:[x-r,x+r]\in \cal{F}\},\]
	inductively,\[{I_k} = [{x_k} - {r_k},{x_k} + {r_k}],{x_k} \in A\backslash \bigcup\limits_{i = 1}^{k - 1} {{I_i}},\ {r_i} > \dfrac{1}{2}\sup \{ r:[x - r,x + r] \in \mathcal{F},\ x \in A\backslash \bigcup\limits_{i = 1}^{k - 1} {{I_i}} \}. \]
	
	If there exists a $N$ such that $A\subset \bigcup\limits_{i = 1}^{N} {{I_i}}$, then the condition is similar to condition 1.
	
	Otherwise  $A\subset \bigcup\limits_{i = 1}^{\infty } {{I_i}}$. This is because $\left\{ {{[x_i-\dfrac{1}{2}r_i,x_i+\dfrac{1}{2}r_i]}} \right\}_{i = 1}^{\infty}$ is pairwise disjoint and then  $r_i \to 0\ \ when\ \  i \to\infty.$ If $a\in A$, there must exists an $r>0$ such that $[a-r,a+r]\in \mathcal{F}$. Then we can find a $r_j<\dfrac{1}{2}r$, thus \[a \in \bigcup\limits_{i = 1}^{j - 1} {[{x_i} - {r_i},{x_i} + {r_i}]} .\]

	For a infinite $N$, we can choose two subcollections ${\mathcal{F}_1}$ and ${\mathcal{F}_2}$ such that $\left\{ {{I_i}} \right\}_{i = 1}^N={\mathcal{F}_1} \cup {\mathcal{F}_2}$ while each ${\mathcal{F}_i}$ consists of disjoint intervals from$\left\{ {{I_i}} \right\}_{i = 1}^N.$  Moreover, if [ ] denotes an element in ${\mathcal{F}_1}$ and ( ) denotes an element in ${\mathcal{F}_2}$, then the chosen ${\mathcal{F}_1} \cup {\mathcal{F}_2}$ consists of finite chains with a same  pattern of 
	( [ ) ( ] [ ) ( ] [ ) ].
	
	Then for the $I_{N+1}$, we have known that $x_{N+1}  \notin \bigcup\limits_{i = 1}^{N } {{I_i}}$ and $r_{N+1}<2r_i$ for each $i<N+1$. Thus $I_{N+1}$ intersects at most 4 intervals in $\bigcup\limits_{i = 1}^{N } {{I_i}}.$ 
	
	Next, \{ \} denotes $I_{N+1}$.
	$\bigcup\limits_{i = 1}^{N } {{I_i}}$ consists of finite chains with a pattern of 
	( [ ) ( ] [ ) ] and there are only two ways of how the leftside of $I_{N+1}$ intersecting some chain in $\bigcup\limits_{i = 1}^{N } {{I_i}}$:
	
	\qquad\qquad\qquad( [ ) ( ] [ ) ( ] [ ) \{ ] \}
	and ( [ ) ( ] [ ) ( ] [ \{ )  ] \}.\\
	
	In the first conditon,  $I_{N+1}$ can be added to ${\mathcal{F}_1}$. In the second condition, $I_{N+1}$ can be added to ${\mathcal{F}_2}$ after remove the [ ] intersecting $I_{N+1}$. 
	
	If $I_{N+1}$ intersects two chains in $\bigcup\limits_{i = 1}^{N } {{I_i}}$, we can exchange the elements of ${\mathcal{F}_1}$ and ${\mathcal{F}_2}$ in some chain, if necessary, to ensure the intervals in every ${\mathcal{F}_i}$ are disjoint after adding $I_{N+1}$ to one of them in a way as above.
	\\
	
	\textbf{Conditon 3:} $A$ is unbounded.\\
	
	From an idea in [17]2.2.8, we can deduce a fact that there is a subcollection ${\mathcal{R}} \subset {\mathcal{F}}$ consisting of at most countable disjoint intervals such that for any $[x-r,x+r]\in \mathcal{F}$, we can find an element $[x_i-r_i,x_i+r_i]\in {\mathcal{R}}$ such that\[[x - r,x + r] \cap [{x_i} - {r_i},{x_i} + {r_i}] \ne \emptyset \ \ and\ \ r < \frac{4}{3}{r_i}.\]
	
	\[{\mathcal{F}_i}=\{[x-r,x+r]\in \mathcal{F},[x - r,x + r] \cap [{x_i} - {r_i},{x_i} + {r_i}] \ne \emptyset \ \ and\ \ r < \frac{4}{3}{r_i}\}
	\] \[A_i=\{x:[x - r,x+ r]\in\mathcal{F}_i\}\]
	
	Since ${\mathcal{R}}=\bigcup\limits_{i = 1}^{\infty } {[{x_i} - {r_i},{x_i} + {r_i}]}$ , then ${\mathcal{F}}=\bigcup\limits_{i = 1}^{\infty } {\mathcal {F}_i}$ and $A=\bigcup\limits_{i = 1}^{\infty } A_i$.\\
	
	Now we consider ${\mathcal {F}_i}$ and $A_i$. Since $A_i$ is bounded, we use the procedure in condition 2 to get a family of intervals $\bigcup\limits_{j = 1}^{\infty } {[{x^i_j} - {r^i_j},{x^i_j} + {r^i_j}]}$ cover $A_i$ and set \[ [{x^i_1} - {r^i_1},{x^i_1} + {r^i_1}]= [{x_i} - {r_i},{x_i} + {r_i}].\]
	
	\!\!\!\!\!\!\!\!\it \textbf{Claim :} There exists a $N_i$ such that ${A^i}$ is covered by $\bigcup\limits_{j = 1}^{N_i} {[{x^i_j} - {r^i_j},{x^i_j} + {r^i_j}]}.$
	
	\!\!\!\!\!\!\!\!proof of claim :     
	\rm Otherwise $r_j^i \to 0\ \ when\ \  j \to\infty $.
	
	If all the intervals of $\bigcup\limits_{j = 2}^{N_i} {[{x^i_j} - {r^i_j},{x^i_j} + {r^i_j}]}$ lie in one side of $[{x_i} - {r_i},{x_i} + {r_i}]$, then \[x_j^i < x_2^i - r_2^i\ \ for\ \ \forall\  j > 2.\]
	
	But  $[{x^i_j} - {r^i_j},{x^i_j} + {r^i_j}]\cap [{x_i} - {r_i},{x_i} + {r_i}] \ne \emptyset$ deduces that ${r^i_j}>{r^i_2}\ \ for\ \ \forall\  j > 2$.\\
	which is a contradiction to $r_j^i \to 0\ \  $
	
	This completes the proof of Claim .\\
	
	From the claim, we can choose the leftmost interval and the rightmost interval together with  $[{x_i} - {r_i},{x_i} + {r_i}]$ to cover $A_i$. ${\mathcal{D}^i}$ denotes these three intervals and then $\bigcup\limits_{i = 1}^{\infty} {\mathcal{D}^i}$ covers $A$. 
	
	Then we can use the  similar debate in condition 2 to prove that $\bigcup\limits_{i = 1}^{\infty} {\mathcal{D}^i}$ equals to two subcollections ${\mathcal{F}_1}$ and ${\mathcal{F}_2}$ of disjoint intervals and obviouly ${\mathcal{F}_1} \cup {\mathcal{F}_2}$ still cover $A$.
\end{proof}

\section*{Acknowledgement}
    The author would like to thank her supervisor and some classmates for their useful suggestions while writing this paper.


\begin{thebibliography}{00}
	
	\bibitem{1}Aldaz,J.M. (2019). Besicovitch type properties in metric spaces. Available at the Mathematics ArXiv.
	
	\bibitem{Al1}Aldaz, J.M.(2021) Kissing Numbers and the Centered Maximal Operator.{\em J Geom Anal}.
	
	\bibitem{1}Aimar,H., Forzani,L. (2001).On the Besicovitch Property for Parabolic Balls. {\em J.real analysis exchange}. 27(1):261-267.
	
	\bibitem{}Besicovitch,A.S. (1945). A general form of the covering principle and relative differentiation of additive
	functions. {\em Proc. Cambridge Philos. Soc.} 41, 103–110 
	
	\bibitem{}Besicovitch,A.S. (1946). A general form of the covering principle and relative differentiation of additive
	functions. II. {\em Proc. Camb. Philos. Soc.} 42, 1–10 
	
	\bibitem{1}Bliedtner,J.,Loeb, P.A. (1992). A reduction technique for limit theorems in analysis and probability theory. {\em J. Arkiv Fr Matematik}. 30(1):25-43.
	
	\bibitem{1}Boyvalenkov,P. (1997). On the Besicovitch constant in small dimensions. {\em J. Comptes rendus de l'Academie bulgare des ences: ences mathematiques et naturelles}. 50:17-18.
	
	\bibitem{1}Donne,E.L., Rigot,S. (2014). Besicovitch Covering Property for homogeneous distances in the Heisenberg groups.{\em J. Journal of the European Mathematical Society}. 19(5).
	
	
	\bibitem{1}Donne,E.L., Rigot S. (2016). Besicovitch Covering Property on graded groups and applications to measure differentiation. {\em J.Journal für die reine und angewandte Mathematik (Crelles Journal) }2019(750).
	
	\bibitem{1}Donne,E.L., Rigot S. (2015). Remarks about the Besicovitch Covering Property in Carnot groups of step 3 and higher. {\em J. Proceedings of the American Mathematical Society}.144(5).
	
	\bibitem{}Evans, L.C., Gariepy,R.F. (2015). Measure Theory and Fine Properties of Functions, Revised Edition[M]. 
	
	\bibitem{1}Füredi,Z., Loeb,P. A. (1994). On the best constant for the Besicovitch covering theorem. {\em J.Proceedings of the American Mathematical Society}. 121(4):1063-1073.
	
	\bibitem{1}Federer,H. (1969). Geometric measure theory, Springer-Verlag, New York.
	
	\bibitem{1} Golo,S., Rigot.S. (2019). The Besicovitch covering property in the Heisenberg group revisited. {\em J. The Journal of Geometric Analysis}.29(4):3345–3383.
	
	\bibitem{1}Itoh,T. (2018).The Besicovitch covering theorem for parabolic balls in Euclidean space. {\em J.Hiroshima mathematical journal}. 48(3):279-289.
	
	\bibitem{1}Jost,J., Hong,V.L., Tran,T.T. (2020).Differentiation of measures on complete Riemannian manifolds. Available at the Mathematics ArXiv.
	
	\bibitem{1}Krantz,S.G. (2019). The Besicovitch covering lemma and maximal functions. {\em J. Rocky Mountain Journal of Mathematics}.49(2):539-555.
	
	\bibitem{1}Loeb,P.A. (1989). On the Besicovitch covering theorem. {\em J.SUT Journal of Mathematics}. 25(1).
	
	\bibitem{1}Loeb,P.A. (1993). An optimization of the Besicovitch covering. {\em J.Proceedings of the American Mathematical Society}.118(3):715-716.
	
	\bibitem{1}Loeb,P.A.  (1997).Opening covering theorems of Besicovitch and Morse. {\em J Mathematica Moravica}. Special volume:3-11.
	
	\bibitem{1}Loeb,P.A, Talvila,E. (2001). Covering theorems and Lebesgue integration. {\em J.Scientiae Mathematicae Japonicae}.
	
	\bibitem{1}Morse,A.P. (1947). Perfect blankets. {\em J.Trans. Amer. Math. Soc}. 61, 418–442.
	
	
	\bibitem{1} Preiss,D. (1981). Dimension of metrics and differentiation of measures, General topology and its relations to modern analysis and algebra, V.	
	
	\bibitem{1}Rigot,S. (2004). Counter example to the Besicovitch covering property for some Carnot groups equipped with their Carnot-Carathéodory metric. {\em J. Mathematische Zeitschrift}. 248(4):827-848.
	
    \bibitem{1}Sullivan,J.M. (1994). Sphere packings give an explicit bound for the Besicovitch Covering Theorem. {\em J.Journal of Geometric Analysis}.4(2):219.	
	
	
		
	
\end{thebibliography}
\end{document}